\documentclass[12pt]{article}

\usepackage{latexsym,amsfonts,amsmath,
amssymb,epsfig,tabularx,amsthm,dsfont,enumerate}

\usepackage{graphicx,pstricks-add}

\usepackage[bookmarks=false]{hyperref}

\makeatletter
\newcommand{\LeftEqNo}{\let\veqno\@@leqno}
\makeatother

\theoremstyle{definition}
\newtheorem{thm}{Theorem}
\newtheorem{rem}[thm]{Remark}
\newtheorem{lem}[thm]{Lemma}

\newcommand\R{{\mathbb{R}}}

\renewcommand\natural{\mathbb{N}}

\newcommand\N{\mathbb{N}}

\newcommand\eu{{\rm e}}

\newcommand{\e}{\varepsilon}
\newcommand{\wt}{\widetilde}
\newcommand{\widebar}[1]{\mbox{\kern1pt\hbox{
\vbox{\hrule height 0.5pt \kern0.25ex
        \hbox{\kern-0.05em \ensuremath{#1 }\kern-0.05em}}}}\kern-0.1pt}
\newcommand\il{\left<}
\newcommand\ir{\right>}

\newcommand{\dint}{\, {\rm d} }

\newcommand{\abs}[1]{\left\vert #1 \right\vert}
\newcommand{\norm}[1]{\left\Vert #1 \right\Vert}

\newcommand{\F}{\mathcal{F}}
\newcommand{\M}{\mathcal{M}}

\newlength{\fixboxwidth}
\renewcommand{\rho}{{\varrho }}

\title{\vspace{-0.5cm}
Reproducing Kernels of Sobolev Spaces on $\R^d$
\\
and Applications to \\
Embedding Constants and Tractability
}

\author{Erich Novak\footnote{
Partially supported by the DFG-Priority Program 1324.},
\\
Mathematisches Institut, Universit\"at Jena\\
Ernst-Abbe-Platz 2, 07743 Jena, Germany\\
email: erich.novak@uni-jena.de\\
Mario Ullrich
\\
Institut f\"ur Analysis, Johannes Kepler Universit\"at\\
Linz, Austria
\\
email: mario.ullrich@jku.at\\
\qquad
Henryk Wo\'zniakowski\footnote{Partially
supported by NSC, Poland,
DEC-21013/09/B/ST1/04275.}\\
Department of Computer Science, Columbia University,\\
New York, NY 10027, USA, and\\
Institute of Applied Mathematics, University of Warsaw\\
ul. Banacha 2, 02-097 Warszawa, Poland\\
email:\ henryk@cs.columbia.edu\\
Shun Zhang\footnote{Partially
supported by the NNSF of China
(11301002), Anhui Provincial Natural Science Foundation (1408085QF107)
and Talents Youth Fund of Anhui Province Universities
(2013SQRL006ZD).
}\\
School of Computer Science and Technology, Anhui
University,\\
 Hefei 230601, China\\
 email:\ shzhang27@163.com}
\begin{document}

\maketitle

\begin{abstract}
The standard Sobolev space
$W^s_2(\R^d)$, with arbitrary positive integers $s$ and $d$ for which $s>d/2$,
has the reproducing kernel
$$
K_{d,s}(x,t)=\int_{\R^d}\frac{\prod_{j=1}^d\cos\left(2\pi\,(x_j-t_j)u_j\right)}
{1+\sum_{0<|\alpha|_1\le s}\prod_{j=1}^d(2\pi\,u_j)^{2\alpha_j}}\,{\rm
  d}u
$$
for all $x,t\in\R^d$,
where $x_j,t_j,u_j,\alpha_j$ are components of $d$-variate $x,t,u,\alpha$, and
$|\alpha|_1=\sum_{j=1}^d\alpha_j$ with non-negative integers
$\alpha_j$.
We obtain a more explicit form for the reproducing kernel $K_{1,s}$
and find a closed form for the kernel $K_{d, \infty}$.

Knowing the form of $K_{d,s}$, we present applications
on the best embedding constants between the Sobolev space $W^s_2(\R^d)$
and $L_\infty(\R^d)$, and on strong polynomial tractability
of integration with an arbitrary probability density.
We prove that the best embedding constants are exponentially small in
$d$, whereas worst case
integration errors of algorithms using $n$ function values
are also exponentially small in $d$ and decay at least
like $n^{-1/2}$.
This yields strong polynomial tractability in the
worst case setting for the absolute error criterion.

\end{abstract}

\section{Introduction and results}  \label{S1}

One of the most studied spaces in mathematical analysis are
Sobolev spaces $W^s_p(\Omega)$ for a positive integer $s$, $p\in[1,\infty]$
and $\Omega\subseteq \R^d$. In this paper we consider $p=2$ and
$\Omega=\R^d$ for arbitrary integers $s$ and $d$. Then $W^s_2(\R^d)$
is a separable Hilbert space equipped with the inner product
\begin{equation}\label{normsob}
\il f,g\ir_{W^s_2(\R^d)}=\sum_{|\alpha|_1\le s}
\il D^\alpha f,D^\alpha g\ir_{L_2(\R^d)}\ \ \ \
\mbox{for all}\ \ f,g\in W^s_2(\R^d).
\end{equation}
Here, $D^\alpha$ is the differential operator
$$
D^{\,\alpha} f(x)=\frac{\partial^{\,|\alpha|}}{\partial\,x_1^{\alpha_1}\,
\partial\,x_2^{\alpha_2}\,\cdots\,\partial\,x_d^{\alpha_d}}\ f(x)
\ \ \ \ \mbox{for all}\ \ x=(x_1,x_2,\dots,x_d)\in\R^d
$$
and $L_2(\R^d)$ is the standard space of square integrable functions
with the inner product
$$
\il f,g\ir_{L_2(\R^d)}=\int_{\R^d}f(x)\,\overline{g(x)}\ {\rm d}x.
$$
The embedding condition $s>d/2$ implies
that we can treat $W^s_2(\R^d)$ as a space of continuous functions and
function values are continuous linear functionals.
This means that
$W^s_2(\R^d)$ is a reproducing kernel Hilbert space with a
reproducing kernel $K_{d,s}$, i.e.~$W^s_2(\R^d)=H(K_{d,s})$.
This is a function defined on
$\R^d\times \R^d$ such that $K_{d,s}(\square, t)\in W^s_2(\R^d)$ for
all $t\in \R^d$, the matrix $(K_{d,s}(x_k,x_j))_{k,j=1,2,\dots,n}$
is hermitian and semi-positive definite for all choices of $n$ and
$x_j\in\R^d$, and most importantly
$$
f(t)=\il f,K_{d,s}(\square,t)\ir_{W^s_2(\R^d)}\ \ \ \
\mbox{for all}\ \ f\in W^s_2(\R^d)\ \ \mbox{and for all}\ \
t\in \R^d.
$$
Here, $\square$ is used as the placeholder for the variable of a
function  we consider.
Sometimes we use the shorter notation
$\delta_t(x) = K(x,t)$, hence
$$
f(t) = \il f , \delta_t \ir \qquad \mbox{for all} \quad f \in H(K).
$$

The knowledge of the reproducing kernels is very useful in the
analysis of many computational problems.
Examples include multivariate integration and approximation,
scattered data approximation, statistical and
machine learning, the numerical solution of partial differential
equations, see for instance \cite{BTA04,B03,FY11,FY13,NW08,
SW06,Ta96,Wa90,We05}.
We only mention one application of the kernel $K$ to
the best linear estimation (or optimal recovery or Kriging).
The problem is to find $f \in H$ with minimal norm
such that $f(x_i) = y_i$ for $i=1, 2, \dots , n$.
The solution is an abstract spline of the form
$f^* = \sum_{j=1}^n \alpha_j \delta_{x_j}$,
where $\alpha_j$'s are chosen such that $f^*(x_i)=y_i$
for $i=1, 2, \dots , n$.

It is usually enough to analyse
reproducing kernels instead of the corresponding Hilbert spaces.
It is therefore somehow surprising that it is difficult to find in the
literature explicit formulas for the reproducing kernels of the Sobolev
spaces $W^s_2(\R^d)$
except for the univariate case $d=1$ with $s=1$ and $s=2$.
For $s=1$, we have
$$
K_{1,1}(x,t)=\tfrac12\,\exp\left(-|x-t|\right)\ \ \ \
\mbox{for all} \ \ x,t\in\R,
$$
see for example \cite{Ta96}, and for $s=2$ we have
$$
K_{1,2}(x,t) = \frac{\sqrt 3}3 e^{-|x-t|\sqrt 3/2}
\sin\left(\frac{|x-t|}2+\frac\pi 6\right)\ \ \ \
\mbox{for all} \ \ x,t\in\R,
$$
see \cite{FY11}.

We want to add that reproducing kernels of the Sobolev spaces
with an equivalent norm to~\eqref{normsob} or reproducing kernels
of generalized Sobolev spaces can be found in the literature,
see for instance \cite{FY11,FY13,SW06,We05}. We will return to
this point later.

The definition of the Sobolev space $W^s_2 (\R^d)$ makes sense even for
infinite smoothness $s=\infty$, hence we take, in the definition
\eqref{normsob} of the norm, all partial derivatives
of any order.
Observe that this is now a tensor product Sobolev space.
This and more general Sobolev spaces of infinite order were
studied by Dubinskij~\cite{Du86}.

We comment what we mean by explicit formulas of the
kernels $K_{d,s}$. It is well known that reproducing kernels
are related to complete orthonormal basis's of their corresponding
Hilbert spaces. We illustrate this point for the space $W^s_2(\R^d)$.
Let $\{e_k\}_{k=1}^\infty$ be its complete orthonormal basis. Since
$K_{d,s}(\square,t)\in W^s_2(\R^d)$ for all $t\in \R^d$ then
$$
K_{d,s}(\square,t)=\sum_{n=1}^\infty\il
K_{d,s}(\square,t),e_k\ir_{W^s_2(\R^d)}\,e_k
=\sum_{k=1}^\infty \overline{e_k(t)}\,e_k.
$$
Hence,
$$
K_{d,s}(x,t)=\sum_{n=1}^\infty \overline{e_k(t)}\,e_k(x)\ \ \ \
\mbox{for all}\ \ x,t\in \R^d.
$$
We hope that the reader would agree with us that the last formula is
not very explicit
and more explicit formulas of the reproducing kernels
$K_{d,s}$ are indeed needed.

The following theorem is essentially from Hegland and Marti \cite{HM86}
in the case of finite smoothness $s$.
In fact, it is only one sentence on
page 608 in their paper that the kernel is the Fourier
transform of the rational function
given by (7) on page 614
without even giving the formula for the kernel.
Therefore we give a complete presentation here.

\begin{thm}\label{thm1}
The reproducing kernel of $W^s_2(\R^d)$ with $s>d/2$ is
$$
K_{d,s}(x,t)=\int_{\R^d}\frac{\exp\left(2\pi\,\mathrm{i}\,(x-t)\cdot u\right)}
{1+\sum_{0<|\alpha|_1\le s}\prod_{j=1}^d(2\pi\,u_j)^{2\alpha_j}}\,{\rm
  d}u
\ \ \ \ \mbox{for all}\ \ x,t\in\R^d,
$$
where $x_j,t_j,u_j$ are components of $x,t,u\in\R^d$,
$\mathrm{i}=\sqrt{-1}$,
and $(x-t)\cdot u=\sum_{j=1}^d(x_j-t_j)u_j$ is the usual Euclidean
inner product over $\R^d$.

In the case of infinite smoothness $s= \infty$,
we obtain the kernel
$$
K_{d, \infty} (x,t) =
\prod_{j=1}^d \frac{2}{\pi(x_j-t_j)^3}
\left(  \sin (x_j-t_j) - (x_j-t_j) \cos
  (x_j-t_j) \right)
\ \ \ \ \mbox{for all}\ \ x,t\in\R^d.
$$
\qed
\end{thm}

We obtain these formulas by using the Fourier transform and a few
of its standard properties. In particular, we find a formula which
relates the inner products of $W^s_2(\R^d)$ and $L_2(\R^d)$.
This relation allows us to find a complete orthonormal basis
of $W^s_2(\R^d)$ in terms of a complete orthonormal basis of
$L_2(\R^d)$.

\medskip

We now comment on the form of $K_{d,s}$. Obviously, $K_{d,s}$ takes
real values since we can replace
$$
\exp\left(2\pi\,\mathrm{i}\,(x-t)\cdot u\right)=
\cos\left(2\pi(x-t)\cdot u\right)\ +\
\mathrm{i}\,\sin\left(2\pi(x-t)\cdot u\right)
$$
and the integral of the imaginary part $\sin(2\pi(x-t)\cdot
u)$ is zero since the corresponding integrand is odd with respect to $u$.
We can do even more. Namely,
\begin{eqnarray*}
e^{2\pi\,\mathrm{i}\,(x-t)\cdot u}&=&\prod_{j=1}^d
e^{2\pi\,\mathrm{i}\,(x_j-t_j)u_j}=
\prod_{j=1}^d\left(\cos\left(2\pi(x_j-t_j)u_j\right)\,+\,\mathrm{i}\,
\sin\left(2\pi(x_j-y_j)u_j\right)\right)\\
&=&\sum_{(\beta_1,\beta_2,\cdots,\beta_d)\in\{0,1\}^d}
\prod_{j=1}^d
\left[\cos\left(2\pi(x_j-y_j)u_j\right)\right]^{\beta_j}\,
\left[\mathrm{i}\,\sin\left(2\pi(x_j-y_j)u_j\right)\right]^{1-\beta_j}
\end{eqnarray*}
and all terms with $\beta_j=0$ for some $j$
will disappear after integration as an odd function of~$u_j$.
Therefore, we can rewrite $K_{d,s}$ for all $x,t\in\R^d$ as
\begin{equation}\label{equiker}
K_{d,s}(x,t)
=\int_{\R^d}\frac{\cos\left(2\pi\,(x-t)\cdot u\right)}
{1+\sum_{0<|\alpha|_1\le s}\prod_{j=1}^d(2\pi\,u_j)^{2\alpha_j}}\,{\rm
  d}u
=\int_{\R^d}\frac{\prod_{j=1}^d\cos\left(2\pi\,(x_j-t_j)u_j\right)}
{1+\sum_{0<|\alpha|_1\le s}\prod_{j=1}^d(2\pi\,u_j)^{2\alpha_j}}\,{\rm
  d}u .
\end{equation}

Clearly,
$K_{d,s}(x,x)$ is independent of $x$ and
$$
K_{d,s}(x,x)=
\int_{\R^d}\frac1{1+\sum_{0<|\alpha|_1\le s}
\prod_{j=1}^d(2\pi\,u_j)^{2\alpha_j}}\,{\rm d}u.
$$
Note that $K_{d,s}(x,x)<\infty$ iff $s>d/2$.
This shows the importance of the embedding condition for the
existence of the reproducing kernel.

\medskip

Obviously, it would be useful to find an even more explicit form of
$K_{d,s}$ than that presented in Theorem \ref{thm1}. Ideally, we would
like to find a closed form for the integral defining $K_{d,s}$. We
succeeded with this problem only for $d=1$. In this case, the integral
over $\R$ can be explicitly computed by the residual method and
for all $x,t\in \R$ we obtain
\begin{equation}\label{expkernel}
K_{1,s}(x,t)=-
\sum_{j=1}^s\frac{e^{-|x-t|\,\sin(j\pi/(s+1))}}{s+1}\,
\sin\left(\frac{j\pi}{s+1}\right)\,\cos\left(|x-t|\,\cos\left(\frac{j\pi}{s+1}
\right)\,+\,
\frac{2j\pi}{s+1}\right).
\end{equation}
It is interesting that, for
fixed $x$ and $t\to\infty$, the function $K_{1,s}(x,t)$
decays exponentially for all
$s<\infty$ but only polynomially for $s=\infty$,
see Theorem~\ref{thm1}.
The proofs of all these formulas are provided in Section 2.

In Section~3 we present two applications
based on the form of the reproducing kernel.
The first application is on the best embedding constants between the
Sobolev spaces $W^s_2(\R^d)$ with $s>d/2$ and $L_\infty(\R^d)$.
It is easy to show that the best embedding constant is $K_{d,s}(0,0)^{1/2}$
and it is exponentially small in $d$.

The second application is on integration problems
$$
S_{\rho_d} (f) = \int_{\R^d} f(x) \rho_d (x) \, {\rm d} x
$$
for $f \in W^s_2(\R^d)$ and a probability density $\rho_d : \R^d \to \R^+_0$,
where
$s > d/2$.
We prove that worst case
integration errors of some algorithms that use $n$ function values
is exponentially small in $d$ and decay at least as $n^{-1/2}$. This implies
strong polynomial tractability of integration
for the absolute error criterion.
In addition, we also consider strong polynomial tractability of
integration for tensor product Sobolev spaces.

The final Section~\ref{S4} of this paper
contains concluding remarks on
how the results can be generalized to  weighted Sobolev spaces,
Sobolev spaces with equivalent norms, as well as more general
reproducing kernels Hilbert spaces.

\section{Proofs}  \label{S2}

We will be using standard properties of the Fourier transform which
can be found, for example, in \cite{SW71}.
For integrable functions $f$ over $\R^d$, the Fourier transform is
defined as
$$
[\F f](z)=\int_{\R^d}f(u)\,e^{-2\pi\,\mathrm{i}\,z\cdot u}\,{\rm d}u
\ \ \ \ \mbox{for all}\ \ z\in\R^d,
$$
where, as before, $z\cdot u=\sum_{j=1}^dz_ju_j$ for components $z_j,u_j$
of $z$ and $u$.

It is well known that for $f,g\in L_1(\R^d)\cap L_2(\R^d)$ we have
$$
\il f,g\ir_{L_2(\R^d)}=\il \F f,\F g\ir_{L_2(\R^d)}.
$$
Since $L_1(\R^d)\cap L_2(\R^d)$ is a dense subset of $L_2(\R^d)$
   there is a unique extension of $\F$ to $L_2(\R^d)$.
For simplicity we denote this extension also by $\F$. The mapping $\F$
is an isometry and
$$
[\F^{-1} f](z)=[\F f](-z)\ \ \ \
\mbox{for all} \ \ f\in L_2(\R^d)\ \ \mbox{and}\ \ z\in \R^d.
$$
For $f\in W^s_2(\R^d)$, we have $D^\alpha f\in L_2(\R^d)$ for all
$|\alpha|\le s$. It is known that
$$
[\F (D^\alpha f)] (z) =
\left(\prod_{j=1}^d(2\pi\,\mathrm{i}\,z_j)^{\alpha_j}\right)\
[\F f](z)\ \ \ \ \mbox{for all} \ \ z\in \R^d.
$$
Let
\begin{equation}\label{vfun}
v_{d,s}(z)\,=\,\left(1+\sum_{0<|\alpha|_1\le s}\prod_{j=1}^d(2\pi
  z_j)^{2\alpha_j}\right)^{1/2}\ \ \ \ \mbox{for all} \ \ z\in\R^d.
\end{equation}
Clearly, $v_{d,s}\ge1$ and it is easy to verify that $s>d/2$ implies
$v_{d,s}^{-1}\in L_2(\R^d)$.
We are ready to prove the following lemma which relates the inner
products of $W^s_2(\R^d)$ and $L_2(\R^d)$.
\begin{lem}\label{lem1}
$\qquad$
\begin{eqnarray*}
f\in W^s_2(\R^d)\ \ \ \ &\mbox{iff}&\ \ \ v_{d,s}\,\F f\in L_2(\R^d),\\
\il f,g\ir_{W^s_2(\R^d)} &=& \il v_{d,s}\,\F f,v_{d,s}\,\F g\ir_{L_2(\R^d)}
\ \ \ \ \mbox{for all}\ \ f,g\in W^s_2(\R^d).
\end{eqnarray*}
\end{lem}
\begin{proof}
For $f\in W^s_2(\R^d)$ we have
\begin{eqnarray*}
\|f\|^2_{W^s_2(\R^d)}&=&
\sum_{|\alpha|_1\le s}\|D^\alpha f\|^2_{L_2(\R^d)}=
\sum_{|\alpha|_1\le s}\|\,[\F(D^\alpha)] f\, \|^2_{L_2(\R^d)}\\
&=&
\int_{\R^d}\sum_{|\alpha|_1\le s}\prod_{j=1}^d(2\pi z_j)^{2\alpha_j}\,
|[\F f](z)|^2\,{\rm d}z=
\int_{\R^d} v_{d,s}(z)^2\,|[\F f](z)|^2\,{\rm d}z\\
&=&
\|v_{d,s}\,[\F f]\|^2_{L_2(\R^d)}.
\end{eqnarray*}
This means that $f\in W^s_2(\R^d)$ implies that $v_{d,s}\,\F f\in
L_2(\R^d)$.
Of course, if $v_{d,s}\,\F f\in L_2(\R^d)$ then we can reverse our
reasoning and claim that $f\in W^s_2(\R^d)$. This proves the first
part of Lemma \ref{lem1}.

For $f,g\in W^s_2(\R^d)$ we have
\begin{eqnarray*}
\il f,g\ir_{W^s_2(\R^d)}&=&
\sum_{|\alpha|_1\le s} \il D^\alpha f,D^\alpha g \ir_{L_2(\R^d)}=
\sum_{|\alpha|_1\le s} \il\,[\F(D^\alpha)] f,[\F(D^\alpha)] g \ir_{L_2(\R^d)}\\
&=&
\int_{\R^d}\sum_{|\alpha|_1\le s}\left(\prod_{j=1}^d(2\pi z_j)^{\alpha_j}\right)\,
[\F f](z)\,\left(\prod_{j=1}^d(2\pi z_j)^{\alpha_j}\right)\,
\overline{[\F g](z)}\,{\rm d}z\\
&=&
\int_{\R^d} v_{d,s}(z)^2\,[\F f](z)\,
\overline{[\F g](z)}\,{\rm d}z\\
&=&\il v_{d,s}\,\F f,v_{d,s}\,\F g\ir_{L_2(\R^d)},
\end{eqnarray*}
as claimed in the second part of Lemma \ref{lem1}.\\
\end{proof}

{}From Lemma \ref{lem1} it is easy to find a complete orthonormal
basis of $W^s_2(\R^d)$ in terms of a complete orthonormal basis
$\{ e_k\}_{k=1}^\infty$ of the space $L_2(\R^d)$. Indeed, let
$$
f_k=\F^{-1} \left(v_{d,s}^{-1}e_k\right)\ \ \ \mbox{for all}\ \ k \in \N .
$$
Then $f_k\in L_2(\R^d)$ and $e_k=v_{d,s}\,\F f_k \in L_2(\R^d)$.
Due to the first point of
Lemma \ref{lem1} we also have that $f_k\in W^s_2(\R^d)$. Clearly,
due to the second point of Lemma \ref{lem1}  we have
$$
\il f_k,f_j\ir_{W^s_2(\R^d)}= \il v_{d,s}\,\,\F f_k,
v_{d,s}\,\F f_j\ir_{L_2(\R^d)}=\il e_k,e_j\ir_{L_2(\R^d)}=\delta_{k,j}.
$$
Hence $\{f_k\}_{k=1}^\infty$ is orthonormal in $W^s_2(\R^d)$.

To show that the $\{f_k\}_{k=1}^\infty$ is complete, take an arbitrary
$f\in W^s_2(\R^d)$. Then $v_{d,s}\,\F f\in L_2(\R^d)$ and
\begin{eqnarray*}
v_{d,s}\,\F f&=&\sum_{k=1}^\infty\il v_{d,s}\,\F f,e_k\ir_{L_2(\R^d)}\,e_k=
\sum_{k=1}^\infty\il
f,\F^{-1}(v_{d,s}^{-1}\,e_k)\ir_{W^s_2(\R^d)}\,e_k\\
&=&
\sum_{k=1}^\infty\il f,f_k\ir_{W^s_2(\R^d)}\,e_k.
\end{eqnarray*}
Hence
$$
f=\sum_{k=1}^\infty\il
f,f_k\ir_{W^s_2(\R^d)}\,\F^{-1}(v_{d,s}^{-1}\,e_k)=
\sum_{k=1}^\infty\il f,f_k\ir_{W^s_2(\R^d)}\,f_k,
$$
as claimed.

\medskip

Due to $s>d/2$ we know that $W^s_2(\R^d)$ is a reproducing kernel
Hilbert space and its reproducing kernel is denoted by $K_{d,s}$. We
need to show that $K_{d,s}$ satisfies the formula of
Theorem~\ref{thm1}. Since  $K_{d,s}(\square,t)\in W^s_2(\R^d)$,
Lemma~\ref{lem1} yields for all $f\in W^s_2(\R^d)$
\begin{equation}\label{111}
f(t)=\il f, K_{d,s}(\square,t)\ir_{W^s_2(\R^d)} =
\il v_{d,s}\,\F f,v_{d,s}\,\F [K_{d,s}(\square,t)]\ir_{L_2(\R^d)}.
\end{equation}
On the other hand,
\begin{equation}\label{112}
\begin{split}
f(t)&=[\F^{-1}\F f](t)=\int_{\R^d}e^{2\pi\,\mathrm{i}\,t\cdot u}
[\F f](u)\,{\rm d}u \\
&= \int_{\R^d}
v_{d,s}(u)\,[\F f](u)\,v_{d,s}(u)\,\frac{\exp(2\pi\,\mathrm{i}\,t\cdot
  u)}{v_{d,s}^2(u)}\,{\rm d}u\\
&=\il v_{d,s}\,\F f, v_{d,s}\,
\frac{\exp(-2\pi\,\mathrm{i}\,t\cdot\square)}
{v_{d,s}^2}\ir_{L_2(\R^d)}.
\end{split}
\end{equation}
%
{}From \eqref{111} and \eqref{112} we conclude
$$
\F [K_{d,s}(\square,t)](u)= \frac{\exp(-2\pi\,\mathrm{i}\,t\cdot u)}
{v_{d,s}^2}
$$
or equivalently
$$
K_{d,s}(x,t)= \F^{-1}\left[ \frac{\exp(-2\pi\,\mathrm{i}\,t\cdot\square)}
{v_{d,s}^2} \right](x)
$$
almost everywhere.

Since we are dealing with continuous functions, the last relation must
hold for all arguments, i.e.,
$$
K_{d,s}(x,t)= \int_{\R^d}\frac{\exp\left(2\pi\,\mathrm{i}\,(x-t)\cdot u\right)}
{v_{d,s}^2(u)}\,{\rm d}u,
$$
as claimed. This completes the proof of Theorem~\ref{thm1}
for finite $s$.

We turn to the case $s=\infty$.
Again we obtain
$$
K_{d, \infty }(x,t)=
\int_{\R^d}\frac{\prod_{j=1}^d\cos\left(2\pi\,(x_j-t_j)u_j\right)}
{\sum_{|\alpha|_1 < \infty }\prod_{j=1}^d(2\pi\,u_j)^{2\alpha_j}}\,{\rm
  d}u .
$$
Now
$$
\sum_{|\alpha |_1 < \infty} \prod_{j=1}^d  (2 \pi u_j )^{2 \alpha_j} =
\prod_{j=1}^d \left( \sum_{\alpha =0}^\infty (2 \pi u_j )^{2\alpha} \right) .
$$
For $2\pi |u_j| \ge 1$ the last product
is not finite and therefore we need to integrate
only over $[-1/(2\pi), 1/(2\pi) ]$, and we obtain
$$
K_{d, \infty }(x,t)=
\prod_{j=1}^d  \int_{-1/(2\pi)}^{1/(2\pi)}
 (1-4 \pi^2 u^2 ) \cos (2\pi(x_j-t_j) u) \,
{\rm d} u.
$$
Integration by parts yields
\begin{equation}\label{eq:Kinf}
K_{d, \infty} (x,t) =  \prod_{j=1}^d \frac{2}{\pi(x_j-t_j)^3}
\left(  \sin (x_j-t_j) - (x_j-t_j) \cos
  (x_j-t_j) \right).
\end{equation}
Let\footnote{We propose to call the function $\wt K_\infty$
the \emph{Varenna} function since it was found during the discrepancy
workshop in Varenna, Italy, in June 2016.}
$$
\wt K_\infty(x) :=  K_{1, \infty} (x, 0) =
\prod_{j=1}^d \frac{2}{\pi x^3}
\left(  \sin x  - x \cos
x \right)
\ \ \ \mbox{for all}\ \ \ x\in\R,
$$
which
is possibly the ``simplest'' function
in the space $W^\infty_2 (\R)$, in particular,
this is a $C^\infty$ function with small derivatives, see Figure~ \ref{fig1}.
Using the series representation of $\sin$ and $\cos$ we obtain
\begin{eqnarray*}
\wt K_\infty(x)&=&\frac2{\pi}\
\sum_{j=0}^\infty\frac{(-1)^j}{(2j+1)!(2j+3)}
  \,x^{2j}\\
&=&\frac{2}{3\pi}
\left(1-\frac{x^2}{10}+\frac{x^4}{280}-\dots\right).
\end{eqnarray*}
\begin{figure}
\vspace{-1cm}
\newrgbcolor{qqwuqq}{0. 0.39215686274509803 0.}
\begin{center}
\psset{xunit=0.4cm,yunit=18.0cm,algebraic=true,dimen=middle,dotstyle=o,dotsize=5pt 0,linewidth=0.8pt,arrowsize=3pt 2,arrowinset=0.25}
\begin{pspicture*}(-17,-0.06)(17,0.26)
\psaxes[labelFontSize=\scriptstyle,xAxis=true,yAxis=true,Dx=5.,Dy=0.05,ticksize=-2pt 0,subticks=2]{-}(0,0)(-15,-0.06)(15,0.3)
\psplot[linewidth=1.4pt,linecolor=qqwuqq,plotpoints=200]
{-15}{15}{2.0/(3.14*(x)^(3.0))*(SIN(x)-x*COS(x))}
\end{pspicture*}
\end{center}
\vspace{-5mm}
\caption{The function $\wt K_\infty$}\label{fig1}
\end{figure}
The kernel $K_{d,\infty}$ is generated by the function $\wt K_\infty$
since
$$
K_{d,\infty}(x,t)=\prod_{j=1}^d\wt K_\infty(x_j-t_j)\ \ \
\mbox{for all}\ \ \ x,t\in \R^d.
$$
In particular, we obtain
\begin{equation}\label{eq:Kinf00}
K_{d, \infty }(x,x)=
\left(  \frac{2}{3 \pi} \right)^d = (0.2122\dots)^{\,d}.
\end{equation}

\medskip

We still need to prove \eqref{expkernel}. Assume now that $d=1$.
Then
\[
K_{1,s}(x,t) \,=\,
\int_\R \frac{\eu^{2\pi\,\mathrm{i}\, (x-t)u}}
{1+\sum_{\ell=1}^s(2\pi u)^{2\ell}}\,\dint u=
\int_\R \frac{\eu^{2\pi\,\mathrm{i}\, |x-t|u}}
{1+\sum_{\ell=1}^s(2\pi u)^{2\ell}}\,\dint u.
\]
We did not find an explicit formula for the kernel
$K_{1,s}$ in the literature except for $s=1$ and $s=2$.
In any case the derivation of the kernel
is similar as in \cite{Ta96} for $s=1$.

To compute the integral that appears in the formula for $K_{1,s}$
we use the residual method. Let $\xi=|x-t|$.
Then the integrand is
\[
f(u):=\frac{\exp(2\pi i\xi u)}{1+\sum\limits_{\ell=1}^s (2\pi u)^{2\ell}} =
\frac{\exp(2\pi i\xi u) \big((2\pi u)^2-1\big)}{(2\pi u)^{2s+2}-1}\;\;\;\;
\text{for}\;\; (2\pi u)^2\neq 1.
\]
The poles of $f$ in the upper half plane are
\[
u_j=\frac1{2\pi}\exp\Big(\mathrm{i}\,
\frac{j\pi}{s+1}\Big)\;\;\;{\rm for}\;\;j = 1,2,\ldots,s.
\]
Note that $u_0=\frac1{2\pi}$ and $u_{s+1}=-\frac1{2\pi}$
 on the real line are not poles.
The integral is then equal to the product of $2\pi\, \mathrm{i}$
by the sum
of the residues of the
integrand at the poles. We have
\[
\begin{split}
{\rm Res}\,f(u_j)&\ =\ \lim_{u\rightarrow u_j}\,(u-u_j)f(u)\,
=\,\lim_{u\rightarrow
  u_j}\,(u-u_j)\frac{\exp(2\pi\,\mathrm{i}\,\xi u)
\big((2\pi u)^2-1\big)}{(2\pi u)^{2s+2}-1}\\
&\ =\ \lim_{u\rightarrow
  u_j}\,(u-u_j)\frac{\exp(2\pi\,\mathrm{i}\,\xi u)
\big((2\pi u)^2-1\big)}{\big((2\pi u)^{2s+2}-1\big)-\big((2\pi
u_j)^{2s+2}-1\big)}\\
&\ =\ \lim_{u\rightarrow u_j}\,\frac{\big((2\pi u)^2-1\big)
\exp(2\pi\,\mathrm{i}\xi u)}
{\frac{(2\pi u)^{2s+2}-(2\pi u_j)^{2s+2}}{u-u_j}}
=\ \frac{\big((2\pi u_j)^2-1\big)
\exp(2\pi\,\mathrm{i}\,\xi u_j)}{2\pi(2s+2)(2\pi u_j)^{2s+1}}.
\end{split}
\]
This yields
\begin{equation}  \label{kernels}
K_{1,s}(x,t)=\frac{\mathrm{i}}{2s+2}\sum_{j=1}^s
\frac{\exp\left(|x-t|\exp(\mathrm{i}\,\frac{j\pi}{s+1}+\mathrm{i}\,
\frac\pi2)\right)}{\exp\left(\mathrm{i}\,j\pi
\frac{2s+1}{s+1}\right)}\left(\exp\Big(\mathrm{i}\,\frac{2j\pi}{s+1}\Big)-1\right).
\end{equation}

The kernel is real valued and we may write  $\wt K_s(x-t)=K_{1,s}(x,t)$ as
\begin{equation}\label{kernel_triang}
\begin{split}
\wt K_s(t)  & = -\frac{1}{2s+2}\sum_{j=1}^s e^{-|t|\sin\left(\frac{j\pi}{s+1}\right)}
\Biggl(
 \sin \left( |t|\cos\Big(\frac{j\pi}{s+1}\Big) + \frac{3j\pi}{s+1} \right) \\
 &\hspace{6cm}- \sin \left( |t|\cos\Big(\frac{j\pi}{s+1}\Big) +  \frac{j\pi}{s+1} \right)
\Biggr)\\
& = - \frac{1}{s+1}\sum\limits_{j=1}^s e^{-|t|\sin\left(\frac{j\pi}{s+1}\right)}
\sin \left(\frac{j\pi}{s+1}\right) \,
\cos \left( |t|\cos\Big(\frac{j\pi}{s+1}\Big) + \frac{2 j\pi}{s+1} \right).
\end{split}
\end{equation}
This proves \eqref{expkernel}, and completes the proofs of all results
mentioned in the previous section.

\vskip 1pc
We illustrate $K_{1,s}$ for $s=1,2,3,4$. We have
\[
\begin{split}
\wt K_1(t) & = \frac12 e^{-|t|},
\\
\wt K_2(t) & = \frac{\sqrt 3}3 e^{-|t|\sqrt 3/2}\sin\Big(\frac{|t|}2+\frac\pi 6\Big),
\\
\wt K_3(t) & = \frac14 \Big( e^{-|t|}
+ \sqrt 2 e^{-|t|/{\sqrt 2}}\sin \frac{|t|}{\sqrt 2} \Big),\\
\wt K_4(t) & = -\frac25 \left( e^{-|t|\sin\frac{\pi}{5}}
\cos \Big( |t|\cos\frac{\pi}{5} +  \frac{2\pi}5 \Big)
\sin \frac{\pi}5 + e^{-|t|\sin\frac{2\pi}5}
\cos \Big( |t|\cos\frac{2\pi}5 + \frac{4\pi}5 \Big)
\sin \frac{2\pi}5 \right).
\end{split}
\]
The function $\wt K_1$ is positive on $\R$, while
the functions  $\wt K_2, \wt K_3$ and $\wt K_4$ also take negative values.

\begin{rem}
The above formulas, see~\eqref{kernel_triang}, show that the functions
$\wt K_s$ decay exponentially fast for finite $s$, while for $s=\infty$ we only
have quadratic decay.
\end{rem}

Clearly, and as one can see from the formula in Theorem~\ref{thm1},
the value of $K_{1,s}(0,0)=\wt K_s(0)$ is monotonically decreasing with $s$.
Using the explicit formula for $\wt K_1$ from above together with \eqref{eq:Kinf00},
we obtain the following lemma.

\begin{lem}\label{lem:K00}
Let $K_{d,s}$ be the reproducing kernel from Theorem~\ref{thm1}. Then,
for $d=1$ and $s\in\N$, we have
\[
\frac{2}{3\pi}\,=\, K_{1,\infty}(0,0) \,\le\, K_{1,s}(0,0) = \frac 1{s+1}
\frac{\cos \frac{\pi}{2s+2}}{\sin \frac{3\pi}{2s+2}} \,\le\, K_{1,1}(0,0)=
\frac12.
\]
\end{lem}

For finite smoothness $s$, the explicit equality
above was shown by Hegland and Marti
\cite[Corollary 1]{HM86}.
These authors also computed the limit for $s \to \infty$.
See also \cite{WKN08,WYT03} for more representations.

\section{Applications} \label{sec:app}
We briefly discuss two applications for which the knowledge of the
form of the reproducing kernel is very helpful.

\subsection{Embedding constants}

It is well-known that all information
of a reproducing kernel Hilbert space $H(K)$ is given by
the reproducing kernel
$K: D\times D\to \R$
of the space.
In particular, one can give an explicit formula for the embedding constant
in $L_\infty(D)$,
i.e., the maximal absolute function value that can be attained
by a function in the unit ball of $H(K)$.
This constant is the norm of the identity operator
$
I_K : H(K) \to L_\infty(D),
$
hence
\begin{equation}  \label{emb-const-general}
\Vert I_K  \Vert \,=\, \sup_{f\neq0}\,
\frac{\|f\|_{L_\infty(D)}}{\|f\|_{H(K)}}.
\end{equation}

The following result is known,
for convenience we  give a short proof.

\begin{lem}\label{lem:emb-const}
Let $H(K)$ be a reproducing kernel Hilbert space with reproducing
kernel
$K:D\times D\to \R$ for a nonempty $D\subseteq\R^d$.
For the embedding $I_K : H(K) \to L_\infty$ we have
\[
\Vert I_K \Vert \,=\, \sup_{x\in D}\, K(x,x)^{1/2}
\]
and, in particular, $\Vert I_K \Vert \,=\, K(0,0)^{1/2}$ if
$H(K)$ is translation invariant.
\end{lem}

\begin{proof}
We denote by $\delta_x(t)=K(t,x)$ the (representer of the) Dirac functional
in $H(K)$, i.e., $f(x)=\il f, \delta_x \ir_{H(K)}$.
Clearly, $K(x,t)=\il \delta_x, \delta_t \ir_{H(K)}$
and hence $\|\delta_x\|_{H(K)} = K(x,x)^{1/2}$.
{}From this we obtain the formula for the norm of $I_K$.
%
Finally, if $H(K)$ is translation invariant,  we clearly have $K(x,t)=K(x+s,t+s)$
and hence $K(x,x)=K(0,0)$.
This completes the proof.
\end{proof}

The embedding constant is important for many applications and
so we have another reason to know the kernel $K$ of a Hilbert space.
We will discuss one of these applications in the next section.

We now turn
to specific estimates of the embedding constant $\|I_{d,s}\|$
for the Sobolev spaces $W^s_2(\R^d)$ with reproducing kernel $K_{d,s}$,
where $I_{d,s}:=I_{K_{d,s}}$ is the embedding from $W^s_2(\R^d)$ to $L_\infty(\R^d)$.
Recall that the diagonal values of $K_{d,s}$ are given by
\[
K_{d,s}(x,x) \,=\, K_{d,s}(0,0) \,=\, \int_{\R^d}\frac{1}
{1+\sum_{0<|\alpha|_1\le s}\prod_{j=1}^d(2\pi\,u_j)^{2\alpha_j}}\,{\rm d}u,
\]
see Theorem~\ref{thm1}.
We change variables by
$t_j=2\pi u_j$,  and obtain from Lemma~\ref{lem:emb-const} that
$$
\|I_{d,s}\|^2 \,=\, \frac{1}{(2 \pi)^d}
\int_{\R^d} \frac{{\rm d} t}{
1+ \sum_{0 < |\alpha|_1 \le s}\prod_{j=1}^d t_j^{2\alpha_j}} .
$$
We use the multinomial identity for $\ell \in \{ 1, 2, \dots , s \}$
and obtain
$$
1+ \sum_{0 < |\alpha|_1 \le s}\prod_{j=1}^d t_j^{2\alpha_j}\ge
1+ \sum_{0 < |\alpha|_1 \le \ell}\prod_{j=1}^d t_j^{2\alpha_j}\ge
\frac1{\ell!}\left(1+\sum_{j=1}^\ell t_j^2\right)^{\ell}.
$$
Therefore
$$
\|I_{d,s}\|^2  \,\le\,
\frac{ \ell ! }{(2 \pi)^d}
\int_{\R^d} \frac{{\rm d} t}{(1 + \sum_{j=1}^d t_j^2)^\ell}
\,=\,
\frac{ \ell ! }{(2 \pi)^d}  \, \frac{2 \pi^{d/2}}{\Gamma (d/2)}
\int_0^\infty \frac{y^{d-1}}{(1+y^2)^\ell} \, {\rm d} y .
$$
The last integral is finite iff $2\ell -d \ge 1$.

Let $d=1$. Then $s\ge 1$ and we can take $\ell=1$. Then
$$
\int_0^\infty \frac{y^{d-1}}{(1+y^2)^\ell} \, {\rm d} y
=\int_0^\infty\frac1{1+y^2}\,{\rm d} y=\frac{\pi}2,
$$
and since $\Gamma(1/2)=\sqrt{\pi}$ we have
$$
\|I_{1,s}\|^2\le\frac12\le
\frac{(d+1)^2}{2}\ \frac1{2^d\pi^{d/2}}\bigg|_{d=1}.
$$
For $d\ge1$, we have
\begin{equation}\label{1919}
\int_0^\infty \frac{y^{d-1}}{(1+y^2)^\ell} \, {\rm d} y \le
\int_0^1y^{d-1}\,{\rm d}y\,+\,\int_1^\infty y^{-2\ell+d-1}\,{\rm d}y=
\frac1d+\frac1{2\ell-d}\,.
\end{equation}

Let $d=3$. Then $s\ge 2$ and we can take $\ell=2$. Since
$\Gamma(3/2)=\sqrt{\pi}/2$ then
$$
\|I_{3,s}\|^2\le \frac4{3\pi^2}\le
\frac{(d+1)^2}{2}\ \frac1{2^d\pi^{d/2}}\bigg|_{d=3}.
$$

Assume now that $d$ is even, $d=2k$ with $k\ge1$. Then $s>d/2$ means that
$s \ge k+1$ and we may take $\ell = k+1$. Hence $2 \ell -d = 2$.
Using this value of $\ell$, and remembering that
$\Gamma(d/2)=\Gamma(k)=(k-1)!$
we obtain
$$
\|I_{d,s}\|^2
\,\le\, \frac{2k (k+1)}{2^d \pi^{d/2}}\left(\frac1d+\frac12\right)\,
=\frac{(d+2)^2}4
  \,\frac{1}{2^d \pi^{d/2}}\,\le\,
\frac{(d+1)^2}2
  \,\frac{1}{2^d \pi^{d/2}}.
$$
Assume now that $d$ is odd, $d=2k+1$ with $d\ge5$.
Then $s > d/2$ means that, again,
we may take $\ell = k+1$ and $k\ge2$.
The Gamma function $\Gamma$ is monotone increasing for $x\ge2$.
Therefore $\Gamma(k+1/2)\ge \Gamma(k)=(k-1)!$, and we obtain
$$
\|I_{d,s}\|^2
\,\le\, \frac{2 k (k+1) }{2^d \pi^{d/2}}\,\left(\frac1d+1\right)\,\le\,
\frac{(d+1)^2}2
  \,\frac{1}{2^d \pi^{d/2}}.
$$
Hence, for all $d$ we have
$$
\|I_{d,s}\|^2\le \frac{(d+1)^2}2\,\frac1{2^d\pi^{d/2}}.
$$
Clearly, this means that $\|I_{d,s}\|$ goes exponentially fast to
zero when $d$ approaches infinity. Asymptotically, the speed of
convergence with respect to $d$ is
$(2^{1/2}\pi^{1/4})^{-d}=(0.531\dots)^d<(6/11)^d =(0.545\dots)^d.
$
It can be verified numerically that
for all values of $d$ we have
$\|I_{d,s}\|\le 10.03\,(6/11)^d$.

We can also obtain a lower bound on $\|I_{d,s}\|$.
It is clear that $\|I_{d,s}\|$ is
a decreasing function of $s$ and therefore it is lower bounded
for $s=\infty$, which is $(2/(3\pi))^{d/2}=(0.460\dots)^d
\ge (5/11)^d=(0.454\dots)^d$. We summarize these estimates in the
following theorem.
\begin{thm}\label{thm:emb}
Let $I_{d,s}$ be the embedding from $W^s_2(\R^d)$ to $L_\infty(\R^d)$. Then
for $d,s\in\N$ with $s>d/2$, we have
\[
\left(\frac{5}{11}\right)^{d} \,\le\,\left(\frac{2}{3\pi}\right)^{d/2}
\,=\, \|I_{d,\infty}\|
\,\le\, \|I_{d,s}\|
\,\le\, \frac{d+1}{2^{(d+1)/2} \pi^{d/4}}\
\,\le\, 10.03 \left(\frac{6}{11}\right)^{d} .
\]
\end{thm}

\subsection{Strong polynomial tractability of integration}  \label{subsec:tract}
We now study the integration problem
$$
S_\rho (f) = \int_{D} f(x) \rho (x) \, {\rm d} x
\ \ \ \mbox{for} \ \ \  f \in H(K),
$$
where $H(K)$ is a reproducing kernel Hilbert space of integrable
functions defined on $D\subseteq\R^d$
with kernel $K$, and a probability density $\rho : \R^d \to \R^+_0$,
i.e., $\int_{D} \rho (x) \,  {\rm d} x =1$.

Consider a QMC algorithm
$$
A_n(f) = \frac{1}{n} \sum_{j=1}^n f(x_j)\ \ \ \mbox{for}\ \ \ f\in
H(K)
$$
for some points $x_1, x_2, \dots , x_n \in D$.
It is well known that the worst case error of $A_n$ is
\[\begin{split}
e_K(x_1, x_2, \dots , x_n) \,&:=\,
\sup_{\Vert f \Vert_{H(K)} \le 1 } \bigg| S_\rho(f) -
\frac{1}{n} \sum_{j=1}^n f(x_j) \bigg| \\
&\,=\,
\bigg\Vert \int_{\R^d}
\delta_x\, \rho (x) \, {\rm d} x - \frac{1}{n} \sum_{j=1}^n
\delta_{x_j} \bigg\Vert_{H(K)}
\end{split}\]
with $\delta_x(t)=K(t,x)$.

The function $h = \int_{\R^d} \delta_x \rho(x) \, {\rm d} x$, i.e.,
$$
h(t) = \int_{\R^d} K(t,x) \rho (x) \, {\rm d} x,
$$
is the representer of $S_\rho$, hence $S_\rho (f)= \il f, h \ir$.
It is also known that
if we average the square of the worst case error with respect to
$x_1, x_2, \dots , x_n$ distributed according to the densities $\rho$ then
$$
\int_{D^{n}} e_K^2 (x_1, x_2, \dots , x_n) \rho(x_1) \dots \rho(x_n) \,
{\rm d} x_1 \dots {\rm d} x_n \,\le\,
\frac{1}{n} \int_{\R^d} K (t,t) \rho(t) \, {\rm d } t.
$$
Hence, there exist points $x_1^*, x_2^*, \dots , x_n^* \in  D$ such that
$$
e_K (x_1^*, x_2^*, \dots , x_n^*) \le \frac{1}{\sqrt{n}}
\left( \int_{\R^d} K (t,t) \rho(t) \, {\rm d } t  \right)^{1/2} .
$$
{}From Lemma~\ref{lem:emb-const} we obtain
\begin{equation} \label{eq:error}
e_K  (x_1^*, x_2^*, \dots , x_n^*) \,\le\, \frac{ \Vert I_K \Vert }{\sqrt{n}}.
\end{equation}
Let $n(\e,H(K))$ be the information complexity of integration, i.e.,
the minimal number of function values needed to find an algorithm with
the (absolute) worst case error at most $\e$.
Then~\eqref{eq:error} yields
\begin{equation}\label{1313}
n(\e,H(K))\le
\left\lceil\left(\frac{\|I_K\|}{\e}\right)^2\right\rceil.
\end{equation}
We now apply the last estimates to integration in the space $W^s_2(\R^d)$.
The following theorem follows from~\eqref{eq:error},
\eqref{1313} and Theorem~\ref{thm:emb}.

\begin{thm}
Consider the integration problem $S_\rho$ given by
$$
S_{\rho_d} (f) = \int_{\R^d} f(x) \rho_d (x) \, {\rm d} x\ \ \
\mbox{for}\ \ \ f \in W^s_2(\R^d)
$$
with $s > d/2$ and a probability density $\rho_d$.
There exist
$x_1^*, x_2^*, \dots , x_n^* \in  \R^d$ such that
$$
e_{K_{d,s}} (x_1^*, x_2^*, \dots , x_n^*)
\,\le\, \frac{10.03}{\sqrt{n}} \left(\frac{6}{11}\right)^{d}.
$$
Furthermore,
$$
n(\e,d):=n(\e,W^s_2(\R^d))\le \left\lceil\,100.6009
  \left(\frac{6}{11}\right)^{2d}\,\frac1{\e^2}\right\rceil.
$$
Since $n(\e,d)=\mathcal{O}(\e^{-2})$ with the factor in the big
$\mathcal{O}$ notation independent of $d$, this means that
the  integration problem
is strongly polynomially tractable independently of the probability
densities $\rho_d$'s.
\end{thm}

\begin{rem}
We stress that this positive tractability result holds for the
\emph{absolute error criterion}.
When we use the \emph{normalized error criterion} then we
compare the error with the initial error $\Vert S_{\rho_d}\Vert$ for the
given density $\rho_d$, and consider
$$
n\left(\e\|S_{\rho_d}\|\right)=n\left(\e\|S_{\rho_d}\|,W^s_2(\R^d)\right).
$$
It is well known that $S_{\rho_d}$ is a well defined continuous linear
functional on $W^s_2 (\R^d)$ with
$$
\Vert S_{\rho_d} \Vert^2 =
\int_{\R^{2d}} K_{d,s} (x,t)\, \rho_d(x)\, \rho_d(t) \,
{\rm d } x {\rm d} t
\,\le\, \sup_{x,t\in\R^d} K_{d,s}(x,t)
= K_{d,s}(0,0).
$$
We now show that this bound is optimal for some $\rho_d$.
Indeed, consider an arbitrary continuous
probability density $\rho_d$ on $\R^d$ with compact support
that contains the origin.
Now, with $\rho_{d,\delta}(x):=\rho_d(x/\delta)/\delta$, we obtain
$\lim_{\delta\to0}\Vert S_{\rho_{d,\delta}} \Vert^2=K_{d,s}(0,0)$.

In general, if there exists a number $c\in(0,1]$ such that
$$
\|S_{\rho_d}\|\,\ge\, c\,K_{d,s}(0,0)^{1/2}\ \ \
\mbox{for all}\ \ \ d\in \N
$$
then strong polynomial tractability also holds for the normalized
error criterion.

However, if $\Vert S_\rho \Vert/K_{d,s}(0,0)^{1/2}$ goes to zero with
$d$ approaching infinity then we cannot conclude whether
the integration problem is tractable or not for the normalized error
criterion.
\end{rem}

\begin{rem} \label{rem:mix}
Observe that \eqref{eq:error} holds for arbitrary kernels and
arbitrary probability density functions $\rho$.
In particular, it holds for the tensor product Sobolev spaces
with the kernels
\begin{equation}
\wt K_{d,s} (x,t) = \prod_{j=1}^d K_{1,s} (x_j, t_j)\ \ \
\mbox{for all}\ \ \ x,t\in \R^d .
\end{equation}
In the last formula we can take arbitrary natural numbers $s$ and $d$,
the space is always a space of bounded continuous functions,
and again the embedding constant is given
by $\wt K_{d,s} (0,0)^{1/2}$.
For such spaces and arbitrary densities $\rho_d$  we obtain
the existence of
$x_1^*, x_2^*, \dots , x_n^* \in  \R^d$ such that
$$
e (x_1^*, x_2^*, \dots , x_n^*)
\le \frac{ \Vert \wt  I_{d,s} \Vert }{\sqrt{n}}
   = \frac{ \wt K_{d,s} (0,0)^{1/2} }{\sqrt{n}} =
\frac{K_{1,s} (0,0)^{d/2} }{\sqrt{n}}
   \le \frac{K_{1,1} (0,0)^{d/2} }{\sqrt{n}}
\le \frac{2^{-d/2}}{\sqrt{n}}  .
$$
Again,  all such integration problems are strongly polynomially tractable
for the absolute error criterion.
\end{rem}

\begin{rem}
Some readers might be puzzled since it is well known that,
for example, the integration problem
$$
           S_d (f) = \int_{[0,1]^d} f(x) \, {\rm d} x
$$
is \emph{not} polynomially tractable and
suffers from the curse of dimensionality
for many classical spaces, see \cite{NW08}
for a survey of such results.


In particular, this holds for the spaces
$H^1_{\rm mix}([0,1]^d)=W^1_2 ([0,1])\otimes\dots\otimes W^1_2 ([0,1])$,
i.e., the $d$-fold tensor product of the space $W^1_2 ([0,1])$,
where the domain of functions is restricted to the unit cube $[0,1]^d$.
For a detailed discussion of this, see Chapter~20 of~\cite{NW08}
and papers cited there.
The corresponding space on $\R^d$ is
$H^1_{\rm mix}(\R^d)=W^1_2 (\R)\otimes\dots\otimes W^1_2 (\R)$,
which is the (reproducing kernel) Hilbert space discussed
in Remark~\ref{rem:mix}.
For this space we consider integration with a probability measure
$\rho_d$
  and achieve even strong polynomial tractability.

Observe that a function
$f \in H^1_{\rm mix} ([0,1]^d)$
can always be extended to a function
$\tilde f \in H^1_{\rm mix}(\R^d)$ with $\tilde f\,\big|_{[0,1]^d} = f$.
However, the norms of $f$ and $\tilde f$ can be quite different
and hence from
$f$ being in the unit ball of $H^1_{\rm mix}([0,1]^d)$
we cannot conclude
that $\tilde f$ is in the unit ball of
$H^1_{\rm mix}(\R^d)$. In some sense, the unit ball of
$H^1_{\rm mix}(\R^d)$ is ``quite small'' and admits strong polynomial
tractability whereas the unit ball of $H^1_{\rm mix}([0,1]^d)$
is ``quite large'' and causes the curse of dimensionality.

\end{rem}

\section{Concluding remarks}  \label{S4}

In the final section we discuss a few issues
related to the previous considerations.

\subsection{Sobolev spaces with a different norm}

Due to Lemma \ref{lem1},
we can write the Sobolev space as
$$
 W^s_2(\R^d)=\{f\in L_2(\R^d)\,:\ v_{d,s}\,\F f\in L_2(\R^d)\},
$$
and the norm can be expressed as
$$
\|f\|_{W^s_2(\R^d)}=\|v_{d,s}\,\F f\|_{L_2(\R^d)}.
$$
In \cite{We05}, p.133, the following Sobolev space was used
for a real $s$ with $s>d/2$
$$
H^s(\R^d)=\{f\in L_2(\R^d)\,:\ (1+\|\cdot\|_2^2)^{s/2}\,\hat f\in
L_2(\R^d)\}
$$
with the inner product
$$
\il f,g \ir_{H^s(\R^d)}=(2\pi)^{-d/2}\int_{\R^d}(1+\|u\|^2_2)^s\,\hat f(u)\,
\overline{\hat g(u)}\, {\rm d}u,
$$
and the Fourier transform of $f$ given by
$$
\hat f (x) =(2\pi)^{-d/2}\int_{\R^d} f(u)\,\eu^{-\mathrm{i} x\cdot
  u}\,{\rm d}u,
$$
which differs from $\F$ by a factor depending on  $d$.

Neglecting a slightly different role of the factors,
the basic difference
between $W^s_2(\R^d)$ and $H^s(\R^d)$ is that the function $v_{d,s}$
for the space $W^s_2(\R^d)$ is now replaced by
$v_s=(1+\|\cdot\|_2^2)^{s/2}$
for the space $H^s(\R^d)$. Obviously $v_{d,s}\not=v_s$.
Therefore, although the norms of $W^s_2(\R^d)$ and
$H^s(\R^d)$ are equivalent, they have different reproducing kernels.
Namely, it is proved in \cite{We05} that the reproducing kernel
of $H^s(\R^d)$ is
$$
    K^\ast_{d,s}(x,t)=\frac{2^{1-s}}{(s-1)!}\,\|x-t\|_2^{s-d/2}\,B_{d/2-s}(\|x-t\|_2)
\ \ \ \ \mbox{for all}\ \ x,t\in \R^d,
$$
where $B_{d/2-s}$ is the modified Bessel function of the third kind.

As the reproducing kernel of $W_2^s(\R^d)$,
the reproducing kernel $K^\ast_{d,s}$ of $H^s(\R^d)$ can also
be expressed in the form of an integral,
see \cite[Theorem 10.12]{We05},
\begin{equation}\label{0909}
K^\ast_{d,s}(x,t)= (2\pi)^{-d/2}
\int_{\R^d}\frac{\exp\left(\mathrm{i}\,(x-t)\cdot u\right)}
{\left(1+\sum_{j=1}^d u_j^2\right)^{s}}\,{\rm
  d}u,
\end{equation}
for all  $x,t\in\R^d$,
where $x_j,t_j,u_j$ are components of $x,t,u\in\R^d$.

Note that $K^\ast_{d,s}$ depends on $\|x-t\|_2$, whereas $K_{d,s}$
depends on $x-t$.
Indeed, the norm (and/or scalar product) of $H^s(\R^d)$
is isotropic.
As
the norm for the space $W^s_2(\R^d)$,
the norm of $H^s(\R^d)$ can also be given by $L_2$ norms
of the derivatives.
One can show this as follows,
see also \cite[Section 1.3.5]{SS16} for similar calculations.

Using the formulas
of the Fourier transform for derivatives, we have
$$
\widehat{D^\beta f}(x)=
\hat{f}(x)\cdot \prod_{j=1}^d (\mathrm{i}\cdot x_j)^{\beta_j} ,\ \ \
 \beta\in \mathbb{N}_0^d,
$$
and
\[
\begin{split}
\left(1+\norm{u}_2^2\right)^s &= \sum_{\ell=0}^s \binom{s}{\ell}
\cdot \norm{u}_2^{2\ell}
= \sum_{\ell=0}^s \binom{s}{\ell}\cdot \sum\limits_{\beta \in \mathbb{N}_0^d\atop |\beta|_1=\ell}
\frac{\ell!}{\beta!}\cdot \prod_{j=1}^d u_j^{2\beta_j}\\
&= \sum_{\ell=0}^s \frac{s!}{(s-\ell)!} \sum\limits_{\beta \in \mathbb{N}_0^d\atop |\beta|_1=\ell}
\prod_{j=1}^d  \frac{(\mathrm{i} u_j)^{\beta_j}\overline{(\mathrm{i} u_j)^{\beta_j}}}{\beta_j!},
\end{split}
\]
where $|\beta|_1=\beta_1+\ldots+\beta_d$ and $\beta!=\prod_{j=1}^d (\beta_j!)$.

By Parseval's relation, see Grafakos \cite[Theorem 2.2.14]{LG08}, we obtain
\[
\int_{\R^d} \hat f(u)\,\overline{\hat g(u)}\ {\rm d}u=  \int_{\R^d}f(x)\,\overline{g(x)}\ {\rm d}x.
\]

Hence,
\begin{equation}\label{norm-radial}
\begin{split}
\il f,g\ir_{H^s}
&=  (2\pi)^{-d/2} \sum_{\ell=0}^s \frac{s!}{(s-\ell)!}
   \sum\limits_{\beta \in \mathbb{N}_0^d\atop |\beta|_1=\ell} \frac{1}{ \prod_{j=1}^d (\beta_j!)}
\il D^\beta f,D^\beta g\ir_{L_2(\R^d)}\\
&=  (2\pi)^{-d/2} \sum_{|\beta|_1\le s} \frac{|\beta|_1!}{\beta!} \cdot \binom{s}{|\beta|_1}
\il D^\beta f,D^\beta g\ir_{L_2(\R^d)}.
\end{split}
\end{equation}

We now
estimate the embedding constant $\|I^\ast_{d,s}\|$
for the Sobolev spaces $H^s(\R^d)$ with reproducing kernel $K^\ast_{d,s}$,
where $I^\ast_{d,s}:=I_{K^\ast_{d,s}}$ is
the embedding from $H^s(\R^d)$ to $L_\infty(\R^d)$.
{}From Lemma~\ref{lem:emb-const} and \eqref{0909} we have
$$
\|I^\ast_{d,s}\|^2 \,=\, K^\ast_{d,s}(0,0) \,=\,(2 \pi)^{-d/2}
\int_{\R^d} \left(1+ \|u\|_2^2\right)^{-s} {\rm d} u=
\frac2{2^{d/2}\,\Gamma(d/2)}\,
\int_0^\infty\frac{t^{d-1}}{(1+t^2)^s}\,{\rm d}t.
$$
Note that the embedding constant tends to infinity for
$s \to d/2$. Obviously, if we vary $d$ we must also vary $s=s(d)$ so
that $s(d)>d/2$. Assume that
\begin{equation}\label{beta}
\beta:=\inf_{d\in\natural} (2s(d)-d)>0.
\end{equation}
This assumption allows us to find a bound on $\|I^*_{d,s}\|^2$ only in
terms of $d$. Indeed, we estimate the last integral
by \eqref{1919} with $\ell=s(d)$. Then
$$
\int_0^\infty\frac{t^{d-1}}{(1+t^2)^{s(d)}}\,{\rm d}t\le
\frac1d+\frac1{2s(d)-d}\le1+\frac1{\beta}.
$$
This yields the following bound on $\|I^*_{d,s(d)}\|^2$.
\begin{thm}\label{thm:emb-radial}
Let $I^\ast_{d,s}$ be the embedding from $H^s(\R^d)$ to $L_\infty(\R^d)$
for a real $s=s(d)$ with $s>d/2$ and satisfying \eqref{beta}. Then
$$
\|I^\ast_{d,s}\|^2\le \frac{2(1+1/\beta)}{2^{d/2}\,\Gamma(d/2)}.
$$
\end{thm}
It is well-known that $\Gamma(d/2)$ is super-exponentially large in $d$.
Therefore, the embedding constant for the norm of $H^s(\R^d)$ is
super-exponentially small in $d$, and it is much smaller
than the one for $W^s_2(\R^d)$.

Obviously, we can also consider the integration problem for the space
$H^s(\R^d)$. Then $\eqref{beta}$ implies strong polynomial
tractability of integration in the worst case setting and the absolute
error criterion.

\subsection{Another Sobolev space for $s=\infty$}

For the space $W^s_2(\R^d)$ we can take $s=\infty$, whereas for the
space $H^s(\R^d)$ the choice $s=\infty$ does not make much sense since
$v_s(u)=\infty$ for all $u\not=0$ and the space $H^\infty(\R^d)$
consists only of the zero function.

There are, however, other Sobolev spaces of functions of infinite
smoothness which formally corresponds to $s=\infty$. Furthermore,
the reproducing kernel of such a space
can be given by a properly normalized Gaussian
kernel.
In this section we present two definitions of the norm of such a Sobolev
space that lead to the Gaussian kernel. We do not know
whether these results are
known but we could not find a suitable reference in the
literature.

For $s=\infty$, we can define a Sobolev space by
$$
H^\infty_2(\R^d)=\{f\in L_2(\R^d)\,:\ \eu^{\|\cdot\|^2/4}\,\hat f\in
L_2(\R^d)\}
$$
with the radial symmetric inner product
\begin{equation}\label{product-radial-infty}
\il f,g \ir_{H^\infty_2(\R^d)}=\int_{\R^d}\eu^{\|u\|^2/2}\,\hat f(u)\,
\overline{\hat g(u)}\, {\rm d}u,
\end{equation}
where $\|\cdot\|$ denotes the Euclidean norm and
the Fourier transform $\hat f$
is defined as in the previous section.
Expanding $e^{\|u\|^2/2}$, we obtain
 \begin{equation}\label{product-radial-finite}
\begin{split}
\il f,g \ir_{H^\infty_2(\R^d)}&\,=\,\int_{\R^d}\hat f(u)\,
\overline{\hat g(u)}\,\left(\sum_{\ell=0}^\infty
\frac{\|u\|^{2\ell}}{2^\ell \cdot \ell!}\right)\, {\rm d}u\\
&\,=\, \sum_{\ell=0}^\infty\,
\frac{1}{2^\ell \cdot \ell!}\, \int_{\R^d}\hat f(u)\,
\overline{\hat g(u)}\,\|u\|^{2\ell}\  {\rm d}u  .
\end{split}
\end{equation}
We now use
\[
\norm{u}^{2\ell} =  \sum\limits_{\beta \in \mathbb{N}_0^d\atop
|\beta|_1=\ell} \frac{\ell!}{\beta!}\cdot \prod_{j=1}^d u_j^{2\beta_j}
=  \sum\limits_{\beta \in \mathbb{N}_0^d\atop |\beta|_1=\ell}
\frac{\ell!}{\prod_{j=1}^d (\beta_j!)} \prod_{j=1}^d  (\mathrm{i} u_j)^{\beta_j}\overline{(\mathrm{i} u_j)^{\beta_j}},
\]
as well as $\left(\prod_{j=1}^d
(\mathrm{i}\cdot u_j)^{\beta_j}\right)\hat{f}(u)= \widehat{D^\beta
f}(u)$,  and
\[
\int_{\R^d} \hat f(u)\,\overline{\hat g(u)}\ {\rm d}u=  \int_{\R^d}f(x)\,\overline{g(x)}\ {\rm d}x.
\]

We observe that
\begin{equation}\label{norm-radial2}
\begin{split}
\il f,g\ir_{H^\infty_2(\R^d)}
&\,=\,  \sum_{\ell=0}^\infty
\frac{1}{2^\ell}
\sum_{\beta \in \mathbb{N}_0^d} \frac{1}{ \prod_{j=1}^d (\beta_j!)}
\int_{\R^d} \widehat{D^\beta f}(u)\,\overline{\widehat{D^\beta g}(u)}\ {\rm d}u \\
&\,=\,  \sum_{\beta \in \mathbb{N}_0^d}
 \frac{1}{2^{|\beta|_1}\cdot\prod_{j=1}^d (\beta_j!)}
\il D^\beta f,D^\beta g\ir_{L_2(\R^d)}.
\end{split}
\end{equation}
These inner products are
invariant under orthogonal transformations in the sense that
$$
\abs{\il f,g\ir_{H^\infty_2(\R^d)}}=\abs{\il f\circ O,g\circ
  O\ir_{H^\infty_2(\R^d)}}
$$
for any orthogonal transformation $O\colon\R^d\to\R^d$,
and every $f,g\in H^\infty_2(\R^d)$.
This is easily seen by the formulas \eqref{product-radial-infty} and
\eqref{product-radial-finite}.

We now discuss the reproducing kernel of $H^\infty_2(\R^d)$
with respect to the inner
product (\ref{product-radial-infty}) or \eqref{norm-radial2}.
It is well-known that the function
$$
\delta_0(x) = (2\pi)^{-d/2}\eu^{-\|x\|^2/2}\ \ \ \mbox{for all}\ \ x\in\R^d
$$
is invariant
under the Fourier transform, i.e.,
$$
\hat\delta_0(x)=\delta_0(x) \ \ \ \mbox{for all}\ \
x\in \R^d.
$$
Hence, we obtain
{}from the definition that the Dirac delta $\delta_x, x\in \R^d$ in  $H^\infty_2(\R^d)$ is given by
$$
\delta_x(y) = (2\pi)^{-d/2}\eu^{-\|y-x\|^2/2}\ \ \ \mbox{for all}\ \
x,y\in \R^d.
$$
For this, note that
$$
\hat\delta_x(u) = \eu^{-{\rm i} x u} \hat\delta_0(u) =
\eu^{-{\rm i} x u} \delta_0(u) = (2\pi)^{-d/2} \eu^{-{\rm i} x u}
\eu^{-\|x\|^2/2}
$$
and
$$
\il f,\delta_x \ir_{H^\infty_2(\R^d)}=\int_{\R^d}\hat f(u)\,
\overline{\hat\delta_x(u)}\,\eu^{\|u\|^2/2}\, {\rm d}u=
 (2\pi)^{-d/2} \int_{\R^d}\hat f(x)\,
\eu^{{\rm i} x u} \, {\rm d}u= f(x).
$$

The reproducing kernel of  $H^\infty_2(\R^d)$ is therefore
the famous Gaussian kernel,
$$
K^{\,\rm rad}_{\infty}(x,y) = (2\pi)^{-d/2}\eu^{-\|x-y\|^2/2}
\ \ \ \mbox{for all}\ \ x,y\in \R^d.
$$
Let $I^{\,\rm rad}_{d,\infty}$ be the embedding
{}from $H^\infty_2(\R^d)$ to $L_\infty(\R^d)$. Then for $d\in\N$, we have
\[
 \|I^{\,\rm rad}_{d,\infty}\| =
K^{\rm rad}_{\infty}(0,0)^{1/2} =  (2\pi)^{-d/4}=(0.6316\dots)^d,
\]
which is larger than  $\|I_{d,\infty}\|=(2/(3\pi))^d=(0.4606\dots)^d$.

\subsection{Weighted multivariate Sobolev spaces}


Each variable
of $f\in W^s_2(\R^d)$ plays the same role. If we permute variables
in an arbitrary way then we obtain another function  from $W^s_2(\R^d)$
with the same norm as $f$. This property often leads to the curse of
dimensionality for many computational problems, see \cite{NW08},
in the worst case setting for the normalized error criterion.
That is why it seems
reasonable to treat various  variables and groups of variables
differently. This can be achieved by \emph{weighted} spaces.
We illustrate this concept for weighted Sobolev multivariate spaces.
Let
$$
\lambda=\{\lambda_{d,s,\alpha}\}_{d\in\N, |\alpha|\le s}
$$
be a family of positive numbers.
We then define the space
$W^{s,\lambda}_2(\R^d)$ as the space $W^s_2(\R^d)$ with
the redefined norm by
$$
\|f\|^2_{W_2^{s,\lambda}(\R^d)}=\sum_{|\alpha|_1\le s}
\lambda_{d,s,\alpha}\,\|D^\alpha
f\|^2_{L_2(\R^d)}.
$$
Note that for
\begin{eqnarray*}
\lambda_{d,s,\alpha}=1&\mbox{\ \ we have the space}& \ \
  W^s_2(\R^d),\\
\lambda_{d,s,\alpha}=\frac{|\alpha|_1!}{(2\pi)^{d/2}\,\alpha!}\,
\binom{s}{|\alpha|_1}
&\mbox{\ \ we have the space}&\ \  H^s(\R^d),\\
\lambda_{d,\infty,\alpha}=\frac1{2^{|\alpha|_1}\,\alpha!}&
\mbox{\ \ we have the space}&\ \  H_2^\infty(\R^d).
\end{eqnarray*}

The assumption that all $\lambda_{d,s,\alpha}<\infty$ is essential.
It is done for a good reason since if one of them $\lambda_{d,s,\alpha}=\infty$
and we adopt the convention
that $\infty\cdot 0=0$ then we must assume that
$D^\alpha f=0$ and $f$ must be in the kernel
of $D^\alpha$, i.e., it must be a polynomial of degree at most of
degree $\max(0,\alpha_j-1)$ for each variables $x_j$.
But the only polynomial that belongs to the space
$W^s_2(\R^d)$ is the zero polynomial, and therefore in this case
the whole space
degenerates to the zero element.

For positive and finite $\lambda$,
it is easy to check that the
reproducing kernel of the weighted Sobolev space
$W_2^{s,\lambda}(\R^d)$ is
$$
K_{d,s,\lambda}(x,t)=\int_{\R^d}
\frac{\exp\left(2\pi\,\mathrm{i}\,(x-t)\cdot u\right)}
{\lambda_{d,s,0}+\sum_{0<|\alpha|\le s}\lambda_{d,s,\alpha}\,
\prod_{j=1}^d(2\pi u_j)^{2\alpha_j}}\ {\rm d}u
\ \ \ \ \mbox{for all}\ \ x,t\in \R^d.
$$
In particular, if $\lambda_{d,s,0}=1$ and $\lambda_{d,s,\alpha}=\beta$
for all $\alpha$ with $|\alpha|\in(0,s]$ then
$$
K_{d,s,\lambda}(x,t)=\int_{\R^d}
\frac{\exp\left(2\pi\,\mathrm{i}\,(x-t)\cdot u\right)}
{1+\beta\,\sum_{0<|\alpha|\le s}
\prod_{j=1}^d(2\pi u_j)^{2\alpha_j}}\ {\rm d}u
\ \ \ \ \mbox{for all}\ \ x,t\in \R^d.
$$
 Note that for large $\beta$, the unit ball $\|f\|_{d,s,\lambda}\le1$
must have all $\|D^\alpha f\|_{L_2(\R^d)}$ small for nonzero~$\alpha$.
Clearly, the larger $\beta$ the smaller the unit ball. In general,
for appropriately
chosen $\lambda_{d,s,\alpha}$ we have a chance to break the curse of
dimensionality of many computational problems.

\subsection{More general Hilbert spaces}

For the Sobolev space $W^s_2(\R^d)$ the function $v_{d,s}$ from \eqref{vfun} was
instrumental in obtaining the reproducing kernel. We now show that
it is not a coincidence and a similar analysis can be done for many
functions $\nu$ which generate corresponding reproducing kernel Hilbert
spaces. We now outline this approach. We opt for simplicity and
consider the class of functions $\nu$ defined on~$\R^d$ from the class
$\M$ which is given by
\[
\M \,:=\, \left\{\nu\in C(\R^d):\;
\nu\ge1\; \text{ and }\; \nu^{-1}\in L_2(\R^d) \right\}.
\]
For $\nu\in\M$, consider the space
$\wt H^\nu$ of all $f \in L_1(\R^d) \cap L_2(\R^d)$ such that the $L_2$-norm
of $\nu\,[\F f]\in L_2(\R^d)$ is finite, and with the inner product
\begin{equation}\label{eq:generalH}
\il f, g\ir_\nu \,=\, \il \nu\,[\F f],\nu\,[\F g]\ir_{L_2(\R^d)},
\end{equation}
see also Lemma~\ref{lem1}.
We denote by $H^\nu$ the completion of $\wt H^\nu$.

Using a similar analysis as before it can be checked that
$H^\nu$ is a reproducing kernel Hilbert space and its reproducing
kernel is
$$
K_\nu(x,t) \,=\, \int_{\R^d}
\frac{\eu^{2\pi i (t-x)\cdot u}}{|\nu(u)|^2}\,\dint u\ \ \ \
\mbox{for all}\ \ x,t\in\R^d.
$$
Note that if we take $\nu=v_{d,s}$ then $H^\nu=W^s_2(\R^d)$ and the
formula for $K_\nu$ is the same as in Theorem \ref{thm1}.
Obtaining bounds or even explicit formulas for the embedding constants
and upper bounds for the error of integration can be found in the same way
as it was done in Section~\ref{sec:app}.

\section{Acknowledgment}

Part of this work was done while S. Zhang was visiting
the Theoretical Numerical Analysis group
of Friedrich-Schiller-Universit\"at Jena.
The kind hospitality is greatly appreciated.

\end{document}